\documentclass[a4,11pt]{article}
\usepackage{amsfonts}
\usepackage{amsmath}
\usepackage{amssymb}
\usepackage{amsthm}
\usepackage{graphicx}

\usepackage{hyperref}
\usepackage{comment}

\usepackage{todonotes}
\usepackage{booktabs}

\usepackage{authblk}

\theoremstyle{plain}
\newtheorem{theorem}{Theorem}[section]

\newtheorem{lemma}[theorem]{Lemma}

\theoremstyle{definition}
\newtheorem{definition}[theorem]{Definition}

\newtheorem{remark}[theorem]{Remark}
\newtheorem{notation}[theorem]{Notation}

\newtheorem*{maintheorem}{Theorem \ref{maintheorem}}

\author{Thomas Wennink\footnote{Department of Mathematical Sciences, University of Liverpool, Liverpool L69 7ZL,  United Kingdom, E-mail address: t.n.wennink@liverpool.ac.uk}}
\affil{University of Liverpool}
\title{Counting the number of trigonal curves of genus 5 over finite fields}
\date{}
\begin{document}
\maketitle

\begin{abstract}
\noindent
The trigonal curves of genus 5 can be represented by projective plane quintics that have 1 singularity of delta invariant 1.
Combining this with a partial sieve method for plane curves we count the number of such curves over any finite field.
The main application is that this then gives the motivic Euler characteristic of the moduli space of trigonal curves of genus 5.
\end{abstract}

\section{Introduction.}
Inside the moduli space $\mathcal{M}_5$ of smooth curves of genus 5 there is a subvariety $\mathcal{T}_5$ parameterizing trigonal curves.
The main result of this article is to count the number $|\mathcal{T}_5(\mathbb{F}_q)|$ for all finite fields $\mathbb{F}_q$.
\begin{maintheorem}
	\textit{The number of smooth trigonal curves of genus $5$ over a finite field $\mathbb{F}_q$, weighted by the size of their automorphism group, is given by}
	\[|\mathcal{T}_5(\mathbb{F}_q)|=q^{11}+q^{10}-q^{8}+1.\]
\end{maintheorem}

This is a step forward in the open problem of computing the cohomology of $\overline{\mathcal{M}}_5$, the moduli space of \emph{stable} genus 5 curves.
Consider the stratification of $\overline{\mathcal{M}}_5$ by topological type.
The moduli spaces $\mathcal{M}_{g,n}$ that appear in the stratification satisfy $g\leq5, n\leq10-2g$.
When $g\leq3$ their $S_n$-equivariant counts over finite fields $\mathbb{F}_q$ are known and polynomials in $q$ (see \cite{kisinlehrerg0} for genus 0, \cite{getzlerg1} or \cite{jonasg1} for genus 1, \cite{jonashe} for genus 2 and \cite{jonasg3} for genus 3).
The missing parts are $\mathcal{M}_{4,1}$, $\mathcal{M}_{4,2}$ and $\mathcal{M}_5$.
If their $S_n$-equivariant counts are polynomials as well then arguing as in Section~3 of \cite{bergtomm4} we get the cohomology and Hodge structure of $\overline{\mathcal{M}}_5$.

We expect the point counts of these spaces to be polynomials in $q$ because of Theorem~F in \cite{chenlan}.
There is a conjectural correspondence due to Langlands (see Section~1.3 in \cite{chenren}) that translates this classification to one on motives of proper smooth stacks over $\mathbb{Z}$.
Assuming this correspondence, by Theorem~F, the cohomology of $\overline{\mathcal{M}}_{4,1}$ must be completely of Tate Hodge type while the cohomologies of $\overline{\mathcal{M}}_{4,2}$ and $\overline{\mathcal{M}}_5$ could in principle also have Tate twists of the cusp form motive $S[12]$.
By Poincar\'e duality it is sufficient to check if this motive appears in $H^{11}$.
It would then appear in $H^{0,11}$, which is not possible since $\overline{\mathcal{M}}_{4,2}$ and $\overline{\mathcal{M}}_5$ do not carry holomorphic 11-forms as they are unirational.
Therefore, assuming Langlands' correspondence, we conclude that the cohomologies of $\overline{\mathcal{M}}_{4,1}$, $\overline{\mathcal{M}}_{4,2}$ and $\overline{\mathcal{M}}_5$ are of Tate Hodge type, which implies that their $S_n$-equivariant point counts over finite fields $\mathbb{F}_q$  are polynomials in q.
(And the same is true of $\mathcal{M}_{4,1}$, $\mathcal{M}_{4,2}$ and $\mathcal{M}_5$ by applying the stratification by topological type and the known results for $g\leq3$ that we cited earlier).

We can decompose the moduli space $\mathcal{M}_5$ into moduli of hyperelliptic curves (whose point count over $\mathbb{F}_q$ is known to be $q^{2g-1}$ for arbitrary genus $g$), moduli of trigonal curves (addressed in this article) and the moduli space $\cal{U}$ of curves that admit an embedding in $\mathbb{P}^4$ as a smooth complete intersection of three quadric hypersurfaces.
From our Theorem~\ref{maintheorem}, and assuming Langlands' correspondence, we deduce that the number of points of $\cal{U}$ over $\mathbb{F}_q$ is a polynomial in $q$.

If the count over finite fields is a polynomial then by Theorem~2.1.8 in \cite{HauselRodKatz} we get the motivic Euler characteristic.
Once we have the $S_n$-equivariant motivic Euler characteristic of all spaces $\mathcal{M}_{g,n}$ with $g\leq5, n\leq10-2g$, we can recover the cohomology and Hodge structure of $\overline{\mathcal{M}}_5$.
The advantage of a point count approach over this topological one is that, if we know the result is a polynomial, we can apply certain tricks.
For example we can determine the polynomial by calculating its evaluation at a finite number of finite fields $\mathbb{F}_q$.
And if we can calculate the top half coefficients of $|\overline{\mathcal{M}}_5(\mathbb{F}_q)|$ then by Poincar\'e duality this determines the polynomial.
Neither of these two methods have been applied in this article but they open possibilities for the remaining calculations required to determine the cohomology of $\overline{\mathcal{M}}_5$.
Also, the fact that, if the conjecture holds, the point count is a polynomial suggests that the calculations would not be unreasonably difficult.

Our strategy to count the number of trigonal curves of genus 5 uses the correspondence with projective plane quintics that have 1 singularity of delta invariant 1 and no other singularities.
We apply a sieve method similar to the one used by Bergstr\"om in \cite{jonasg3} to count this number:

We take all plane quintics that have a singularity of delta invariant 1 at a fixed point $P$ and then we add and subtract various loci of quintics with singular points besides $P$.
After each step of this sieve method more singular curves will have been removed exactly once.
Because there are plane quintics with infinitely many singularities, this procedure will not terminate and so it has to be stopped at some point.
We choose to stop after making sure that all curves with 1 to 5 singularities besides $P$ have been removed exactly once.
The computations for this are done in Section~\ref{sievepart}.
For curves with at least 6 singularities besides $P$ we count in Section~\ref{explicitpart} how many curves there are with exactly that many singularities.
We can then correct the count so that all singular curves been removed exactly once.

The results of our counts are combined in Section~\ref{results}, which also has a weblink to a computer program we have written that counts $|\mathcal{T}_5(\mathbb{F}_2)|$ and $|\mathcal{T}_5(\mathbb{F}_3)|$.
The results of this program agree with our main result, Theorem~\ref{maintheorem}.

\subsection*{Acknowledgements}
I am grateful to my master advisor Carel Faber for suggesting me this problem and for explaining the arguments in the introduction to me.
I would like to thank him and my advisor Nicola Pagani for their generous help and support.
I thank Orsola Tommasi for bringing the article by Vakil and Wood to my attention.
I would like to thank both her and Dan Petersen for helpful discussions about the relation between point counts over finite fields and cohomology.

This is a pre-print of an article published in Geometriae Dedicata. The final authenticated version is available online at:\\ \noindent https://doi.org/10.1007/s10711-019-00508-3.

\section{Preliminaries.}
With $k$ we denote a fixed finite field with $q$ elements.
We define $k_i$ to be the finite field extension of $k$ with $q^i$ elements.

In this article the Frobenius map $\mathcal{F}$ is the geometric Frobenius.
On $\mathbb{P}^2(k)$ it is the endomorphism defined by $(x:y:z)\mapsto(x^q:y^q:z^q)$.
\\
\\
We write $\lambda=[1^{\lambda_1},\ldots,v^{\lambda_v}]$ for the partition of $\sum_{i=1}^v i\cdot\lambda_i$ where $i$ appears $\lambda_i$ times.
The weight of a partition $\lambda$ is denoted by $|\lambda|:=\sum_{i=1}^v i\cdot\lambda_i$.

Given two partitions $\mu,\lambda$, we say $\mu\subseteq\lambda$ if $\mu_i\leq\lambda_i$ for all $i$.
\\
\\
Let $X$ be a scheme defined over $k$.
An $n$-tuple $(x_1,\ldots,x_n)$ of distinct subschemes of $X_{\bar{k}}$ is called a \emph{conjugate $n$-tuple} if $\mathcal{F}(x_i)=x_{i+1}$ for $1\leq i<n$ and $\mathcal{F}(x_n)=x_1$, where $\mathcal{F}$ is the Frobenius map.
	
A $|\lambda|$-tuple $(x_1,\ldots,x_{|\lambda|})$ of distinct subschemes of $X_{\bar{k}}$ is called a $\lambda$-tuple if it consists of $\lambda_1$ conjugate $1$-tuples, followed by $\lambda_2$ conjugate $2$-tuples, etc.
For a scheme $X$ defined over $k$ we write $X(\lambda)$ for the set of $\lambda$-tuples of points of $X$.
\\
\\
Let $C$ be a plane quintic defined by the polynomial equation $\sum_{i=0}^{5} F_i(x,y)z^{5-i}$ and let $L,L'$ be lines.
We say $C$ has tangents $L,L'$ at $(0:0:1)$ if $C$ has a singularity of multiplicity $2$ at $(0:0:1)$ and ${L\cdot L'}=F_2(x,y)$.
We say $C$ has tangents $L,L'$ at a point $P$ if there exists a linear transformation $\phi$ such that $\phi(P)=(0:0:1)$ and $\phi(C)$ has tangents $\phi(L),\phi(L')$ at $(0:0:1)$.
\\
\\
A singularity of a projective plane curve has delta invariant 1 if and only if it is an ordinary node or an ordinary cusp.

\section{Plane curves.}
\label{whatwecount}
The main result of this article is to count
\begin{equation}
\label{startcount}
|\mathcal{T}_5(k)|:=\sum_{C/k}\frac{1}{|\text{Aut}_k(C)|},
\end{equation}
where the sum is over representatives of $k$-isomorphism classes of smooth trigonal curves of genus 5 over $k$.

The dual system of a $g^1_3$ on a given trigonal curve of genus $5$ is a $g^2_5$.
This gives a correspondence between trigonal curves of genus $5$ and projective plane quintics that have a singularity of delta invariant $1$ and no other singularities (this follows from the discussion in chapter 3 of \cite{acgh}).

So in Equation (\ref{startcount}) it is equivalent to instead sum over representatives of $k$-isomorphism classes of plane quintics that have a singularity of delta invariant $1$ and no other singularities.
Because every automorphism extends to $\mathbb{P}^2_k$, $\text{Aut}_k(C)$ is a subgroup of $\text{PGL}_3(k)$.
So the automorphism group of $C$ is the stabilizer of $C$ for the action of $\text{PGL}_3(k)$ on $\mathbb{P}^2_k$.
From the orbit-stabilizer theorem we obtain
\begin{equation}
\label{noautom}
\sum_{C/k}\frac{1}{|\text{Aut}_k(C)|}=\sum_{C/k}\frac{1}{|\text{Stab}_k(C)|}=\sum_{C/k}\frac{|\text{Orb}_k(C)|}{|\text{PGL}_3(k)|}=\frac{|T(k)|}{|\text{PGL}_3(k)|}.
\end{equation}
Here $T(k)$ is the space of plane $k$-quintics with exactly 1 singularity, which has delta invariant $1$.
So in order to compute this sum it is sufficient to count plane quintics without having to worry about their automorphism group.

The curves in $T(k)$ can be separated into those which have an ordinary cusp and those which have an ordinary node.
When a node has both tangents defined over $k$ we say it is a split node.
If the tangents form a conjugate pair it is a non-split node.

Consider a curve in $T(k)$ that has a split node.
Since the curve has no other singularities, the node must be defined over $k$.
So we can apply a $k$-linear coordinate change that maps the node to $(0:0:1)$ and its tangents to $\{x=0\}$ and $\{y=0\}$.
This means that we can count the number of curves in $T(k)$ that have a split node by first counting those that have a singularity at $(0:0:1)$ with tangents $\{x=0\},\{y=0\}$, and then dividing by a suitable subset of $\text{PGL}_3(k)$.
We develop some notation for this.

\begin{notation}
	We write $\mathbf{P}$ for the fixed point $(0:0:1)\in\mathbb{P}^2_k$.
\end{notation}

\begin{notation}
	We write $\mathbf{C}\cong\mathbb{P}^{20}_k$ for the space of plane quintics.
\end{notation}

\begin{definition}
	We define $\mathbf{C}_\text{split}$ to be the subset of $\mathbf{C}$ consisting of curves that have a singularity at $\mathbf{P}$ of multiplicity $2$ with tangents $\{x=0\}$ and $\{y=0\}$.
	The subset $\mathbf{D}_\text{split}\subset\mathbf{C}_\text{split}$ by definition consists of curves that have no singularities besides $\mathbf{P}$.
\end{definition}
Let $\text{Stab}_k(\mathbf{P},\{x,y\})$ be the subgroup of $\text{PGL}_3(k)$ that fixes the set $\{xy=0\}$.
The subgroup $\text{Stab}_k(\mathbf{P},\{x,y\})$ is also the subgroup that fixes $\mathbf{D}_\text{split}$, so we obtain
\[
\frac{|\{C\in T(k)\;|\;\text{$C$ has a split node}\}|}
{|\text{PGL}_3(k)|}
=
\frac{|\mathbf{D}_\text{split}|}
{|\text{Stab}_k(\mathbf{P},\{x,y\})|}.
\]
We now introduce a similar notation for non-split nodes and cusps.
\begin{definition}
	Let $\alpha$ be a fixed element of $k_2$ that is not in $k$.
	The set $\mathbf{C}_\text{non-split}\subset\mathbf{C}$ is given by the curves that have a singularity at $\mathbf{P}$ of multiplicity $2$ with tangents $\{x+\alpha y=0\}$ and $\{x+\mathcal{F}(\alpha)y=0\}$.
	The set $\mathbf{C}_\text{cusp}\subset\mathbf{C}$ consists of the curves that have a singularity at $\mathbf{P}$ of multiplicity $2$ with double tangent $\{y=0\}$ and nonzero coefficient for $x^3z^2$.
	The subsets $\mathbf{D}_\text{non-split}\subset\mathbf{C}_\text{non-split}$ and $\mathbf{D}_\text{cusp}\subset\mathbf{C}_\text{cusp}$ are given by the curves that have no singularities besides $\mathbf{P}$.
\end{definition}
Using (\ref{noautom}) we obtain
\begin{align}
\label{stabs}
\begin{split}
|\mathcal{T}_5(k)|
= &
\frac{|T(k)|}{|\text{PGL}_3(k)|}
=
\frac{|\mathbf{D}_\text{split}|}{|\text{Stab}_k(\mathbf{P},\{x,y\})|}
\\ & +
\frac{|\mathbf{D}_\text{non-split}|}{|\text{Stab}_k(\mathbf{P},\{x+\alpha y,x+\mathcal{F}(\alpha)y\})|}
+
\frac{|\mathbf{D}_\text{cusp}|}{|\text{Stab}_k(\mathbf{P},\{y\})|}
.
\end{split}
\end{align}
We now compute the number of elements of the 3 stabilizer subgroups in Equation (\ref{stabs}) of $\text{PGL}_3(k)$.
The matrices that fix $P$, $\{x=0\}$ and $\{y=0\}$ have the form
\[\begin{pmatrix}
a & 0 & 0 \\
0 & b & 0 \\
c & d & 1
\end{pmatrix}\] where $ab\neq0$.
There are $q^2(q-1)^2$ such matrices.
We can also permute the tangents $\{x=0\}$ and $\{y=0\}$ which adds a factor $2$, so 
\begin{equation}
\label{stab_s}
|\text{Stab}_k(\mathbf{P},\{x,y\})|=2q^2(q-1)^2.
\end{equation}

If $\text{char}(k)\neq2$ then the field $k$ has a quadratic nonresidue $r$ and we can take $\alpha$ such that $\alpha^2=r$.
The matrices that fix $P$, $\{x+\alpha y=0\}$ and $\{x+\mathcal{F}(\alpha)y=0\}$ have the form 
\[\begin{pmatrix}
a & rb & 0 \\
b & a & 0 \\
c & d & 1
\end{pmatrix}
\text{ if $\text{char}(k)\neq2$}
,
\quad
\begin{pmatrix}
a & \frac{\alpha\cdot\mathcal{F}(\alpha)\cdot(a+b)}{\alpha+\mathcal{F}(\alpha)} & 0 \\
\frac{a+b}{\alpha+\mathcal{F}(\alpha)} & b & 0 \\
c & d & 1
\end{pmatrix}
\text{ if $\text{char}(k)=2$}
\]
where $a\neq0$ or $b\neq0$.
In either case we obtain $q^2(q^2-1)$ matrices. Permuting the tangents adds a factor $2$, so
\begin{equation}
\label{stab_ns}
|\text{Stab}_k(\mathbf{P},\{x+\alpha y,x+\mathcal{F}(\alpha)y\})|=2(q^4-q^2).
\end{equation}

The matrices that fix $P$ and $\{y=0\}$ have the form
\[\begin{pmatrix}
a & b & 0 \\
0 & c & 0 \\
d & e & 1
\end{pmatrix}\] where $ac\neq0$.
So we obtain
\begin{equation}
\label{stab_c}
|\text{Stab}_k(\mathbf{P},\{y\})|=q^3(q-1)^2.
\end{equation}

\section{The partial sieve method.}
\label{howwecount}
We will compute $|\mathbf{D}_\text{split}|$, $|\mathbf{D}_\text{non-split}|$ and $|\mathbf{C}_\text{cusp}|$ using a sieve method.
We give full details for the case of $|\mathbf{D}_\text{split}|$, the other two cases are similar.

\begin{definition}
	Let $\mathbf{C}'$ be a subset of $\mathbf{C}$ and let $S\subset\mathbb{P}^2(\bar{k})$ be a set of points.
	We define $\mathbf{C}'(S)$ to be the subset of $\mathbf{C}'$ consisting of those curves that have singularities at the points in $S$.
\end{definition}
The idea of the sieve principle is to compute $|\mathbf{C}_\text{split}|$ and then sieve it by adding or subtracting $|\mathbf{C}_\text{split}(S)|$ for various $S\subset\mathbb{P}(\bar{k})-\{\mathbf{P}\}$.
The amount we add or subtract is chosen in such a way that in the end all curves with singularities other than $P$ have been subtracted exactly once from $|\mathbf{C}_\text{split}|$.

This sieve will not terminate because there are curves that have infinitely many singularities, i.e. those with irreducible components of higher multiplicity.
So we choose a positive integer $N$ and only compute the sieve for sets of points $S$ such that $|S|<N$.

For any $k$-curve the conjugates of its singular points are singular as well.
So it is enough to consider those $S$ that consist of $\lambda$-tuples of points for some $\lambda$.
We define
\begin{equation}
\label{sieveeq}
\mathcal{S}_{\text{split},N}:=\sum_{|\lambda|\leq N}\left((-1)^{\sum_i \lambda_i}\cdot\sum_{S\in(\mathbb{P}^2-\{\mathbf{P}\})(\lambda)}|\mathbf{C}_\text{split}(S)|\right).
\end{equation}
(The sum includes the empty partition, which accounts for $|\mathbf{C}_\text{split}|$.)
Now all curves that have at most $N$ singularities besides $\mathbf{P}$ have been subtracted exactly once.
We then have to account for the rest of the curves.
If a curve has a $\lambda$-tuple of singularities besides $\mathbf{P}$, then to ensure it has been subtracted once in total we need to add it $\sigma_N(\lambda)$ times, where
\[\sigma_N(\lambda):=-\sum_{\substack{\mu\subseteq\lambda \\ |\mu|\leq N}}(-1)^{\sum_i \mu_i}\prod_{i=1}^v\binom{\lambda_i}{\mu_i}.\]

\begin{definition}
	Let $\mathbf{C}'\subset\mathbf{C}$ and let $\lambda$ be a partition.
	Here we allow for $\lambda$ to be an infinite partition $[1^{\lambda_1},\ldots,i^{\lambda_i},\ldots]$.
	We define $\mathbf{C}'(\lambda)$ to be the subset of $\mathbf{C}'$ consisting of those curves whose singularities form precisely a $\lambda$-tuple of points.
	We define $\mathbf{C}'(\mathbf{P},\lambda)$ to be the subset of $\mathbf{C}'$ consisting of those curves whose singularities besides $\mathbf{P}$ form precisely a $\lambda$-tuple of points.
\end{definition}
We have
\[|\mathbf{D}_\text{split}|=\mathcal{S}_{\text{split},N}+\sum_{|\lambda|>N}|\mathbf{C}_\text{split}(\mathbf{P},\lambda)|\cdot \sigma_N(\lambda).\]
(Note that this sum is well defined as there are only finitely many $\lambda$ for which $\mathbf{C}_\text{split}(\mathbf{P},\lambda)$ is non-empty.)

\begin{definition}
	Let $C$ be a projective plane curve, we define $\delta_k^1(C)$ to be the number of singularities over $k$ on $C$ of delta invariant $1$.
\end{definition}
Let $[\lambda,1^1]$ be the partition $\mu$ where $\mu_1=\lambda_1+1$ and $\mu_i=\lambda_i$ for $i>1$.
Since $\text{Stab}_k(\mathbf{P},\{x,y\})$ is the group that fixes $\mathbf{C}_\text{split}(\mathbf{P},\lambda)$, we have
\[\frac{|\mathbf{C}_\text{split}(\mathbf{P},\lambda)|}{|\text{Stab}_k(\mathbf{P},\{x,y\})|}=\frac{1}{|\text{PGL}_3(k)|} \sum_{C\in \mathbf{C}([\lambda,1^1])}|\{\text{ord. split nodes over $k$ on $C$}\}|.\]
Now Equations (\ref{stabs}), (\ref{stab_s}), (\ref{stab_ns}) and (\ref{stab_c}) give us
\begin{align}
\label{finalsieveeq}
\begin{split}
|\mathcal{T}_5(k)|
= &
\frac{\mathcal{S}_{\text{split},N}}{2q^2(q-1)^2}
+
\frac{\mathcal{S}_{\text{non-split},N}}{2(q^4-q^2)}
+
\frac{\mathcal{S}_{\text{cusp},N}}{q^3(q-1)^2}
\\
&
+\frac{1}{|\text{PGL}_3(k)|}\sum_{|\lambda|>N}\sum_{C\in \mathbf{C}([\lambda,1^1])}\delta_k^1(C)\cdot\sigma_N(\lambda)
.
\end{split}
\end{align}
The choice of $N$ is made by considering the difficulty of the first 3 summands of this equation versus the difficulty of the last one.
The former is harder for high $N$, the latter is harder for low $N$.
We choose $N$ to be $5$.

\section{The case of at most 5 extra singularities.}
\label{sievepart}
In this section we will compute $\mathcal{S}_{\text{split},5}$, $\mathcal{S}_{\text{non-split},5}$ and $\mathcal{S}_{\text{cusp},5}$.
We start with $\mathcal{S}_{\text{split},5}$.

To find $\mathcal{S}_{\text{split},5}$ we have to compute $|\mathbf{C}_\text{split}(S)|$ for all $S\in(\mathbb{P}^2-\{\mathbf{P}\})(\lambda)$ with $|\lambda|\leq5$.
The space $\mathbf{C}_\text{split}(S)$ can be described by applying $3\cdot|S|$ linear conditions to $\mathbf{C}_\text{split}(\emptyset)\cong\mathbb{A}_k^{15}$.
Because $S$ is a $\lambda$-tuple of points, all linear conditions can be expressed using coefficients over $k$.
The dimension of $\mathbf{C}_\text{split}(S)$ as a $k$-vector space is therefore the same as its dimension as a $\bar{k}$-vector space.
So applying a $\bar{k}$-linear transformation to $S$ does not change the dimension over $k$ of $\mathbf{C}_\text{split}(S)$.
This means we can split up the right hand side of Equation (\ref{sieveeq}) for different types of configurations of points.

For considerations of space we have not written all the conditions for some of the configurations in the following Lemma.
What the conditions are is clear from the fact that all cases are disjoint.
\newpage
\begin{lemma}
\label{sievetypes}
	The spaces $\mathbf{C}_\text{\emph{split}}(S)$ for all $S\in(\mathbb{P}^2-\{\mathbf{P}\})(\lambda)$ with $|\lambda|\leq5$ are given as follows.	
	\[\resizebox{\textwidth}{!}{\begin{tabular}{cc}
		\emph{configuration of} $S\cup\{\mathbf{P}\}$  & $\mathbf{C}_\text{\emph{split}}(S)$ \\
		\midrule
		$\mathbf{P}$ and at least 2 more points on a line that is not a tangent line of $\mathbf{P}$ & $\emptyset$ \\
		$\mathbf{P}$ and at least 4 more points on a line & $\emptyset$ \\
		5 points in $S$ on a line & $\mathbb{A}_k^{4}$ \\
		4 points in $S$ on a line & $\mathbb{A}_k^{17-3\cdot|S|}$ \\
		$\mathbf{P}$ and 3 points on a line and $\mathbf{P}$ and 2 points on another line & $\mathbb{A}_k^{4}$ \\
		$\mathbf{P}$ and 2 more points on a line, 3 other points on another line & $\mathbb{A}_k^{2}$ \\
		$\mathbf{P}$ and 2 more points, one of those two and 3 other points on another line & $\mathbb{A}_k^{3}$ \\
		$\mathbf{P}$ and 2 points on a line, $\mathbf{P}$ and 2 points on another line & $\mathbb{A}_k^{17-3\cdot|S|}$ \\
		$\mathbf{P}$ and 3 more points on a line & $\mathbb{A}_k^{18-3\cdot|S|}$ \\
		$\mathbf{P}$ and 2 more points on a line & $\mathbb{A}_k^{16-3\cdot|S|}$ \\
		6 points on an irreducible conic & $\mathbb{A}_k^{3}$ \\
		general position & $\mathbb{A}_k^{15-3\cdot|S|}$
	\end{tabular}}\]
\end{lemma}
\begin{proof}
	The first two configurations contradict the condition that there is an ordinary node at $\mathbf{P}$.
	We write the details only for one other case.
	(The rest is proved in a similar way.)
	
	Let $S=\{p_1,\ldots,p_5\}$ such that that $p_1,p_2,\mathbf{P}$ are on a line and $p_2,p_3,p_4,p_5$ are on another line.
	Any curve in $\mathbf{C}_\text{split}(S)$ must have the 2 lines as components.
	If the line through $\mathbf{P}$ is not either $\{x=0\}$ or $\{y=0\}$ then $\mathbf{C}_\text{split}(S)=\emptyset$.
	Without loss of generality we may assume the line to be $\{y=0\}$.
	Now $\mathbf{C}_\text{split}(S)$ is isomorphic to the space of cubics that pass through $p_1,p_3,p_4,p_5$ and have multiplicity 1 at $\mathbf{P}$ with tangent $\{x=0\}$.
	If these conditions are independent then $\mathbf{C}_\text{split}(S)\cong\mathbb{A}_k^3$.
	We apply a $\bar{k}$-linear transformation that keeps $\mathbf{P}$ fixed and maps the points in $S$ to $(1:0:0),(0:1:0),(1:1:0),(1:\alpha:0),(1:0:1)$, for some $\alpha\neq0,1$.
	It is now easy to check that the conditions are independent.
\end{proof}

\begin{remark}
	The dimension of the general position case $\mathbb{A}_k^{15-3\cdot|S|}$ illustrates why $N=5$ is a natural choice.
\end{remark}
The following definition will be helpful when counting how often $|\mathbf{C}_\text{split}(S)|$ appears in Equation (\ref{sieveeq}) for each type of configuration of $S\cup\{\mathbf{P}\}$.
\begin{definition}
	For any scheme $X$ defined over $k$ and $w\in\mathbb{Z}_{\geq0}$ we define
	\[\pi_w(X):=\sum_{|\lambda|=w}(-1)^{\sum_i\lambda_i}\cdot |X(\lambda)|.\]
\end{definition}

From Proposition 3.7 in \cite{vakilwood} we find that the inverse Hasse-Weil Zeta function generates $\pi$.
That is, if $X$ is a scheme of finite type over $k$, we have
\[\frac{1}{Z(X;t)}=\sum_{w=0}^\infty\pi_w(X)t^w.\]

Using the fact that $Z(\mathbb{P}^n;t)=\frac{1}{(1-t)(1-qt)\cdots(1-q^nt)}$ and $Z(\mathbb{P}^n-\{\mathbf{P}\};t)=\frac{1}{(1-qt)\cdots(1-q^nt)}$ we deduce
\begin{align}
\label{piiszero}
\begin{split}
\pi_w(\mathbb{P}^n)&=0\quad\text{for}\quad w\geq n+2, \\
\pi_w(\mathbb{P}^n-\{\mathbf{P}\})&=0 \quad\text{for}\quad w\geq n+1.
\end{split}
\end{align}

\begin{lemma}
	\label{splitresult}
	The sieve count for split nodes is
	\[\frac{\mathcal{S}_{\text{\emph{split}},5}}{2q^2(q-1)^2}=\frac{1}{2}(q^{11}+q^{10}).\]
\end{lemma}
\begin{proof}
	We consider the contribution to Equation (\ref{sieveeq}) for $N=5$ of the case where we have 6 points on an irreducible conic.
	A conic is determined by any 5 of its points so by Lemma~\ref{sievetypes} the contribution is given by $|\mathbb{A}_k^{3}|$ times the number of irreducible $k$-conics through $\mathbf{P}$ times $\pi_5(\mathbb{P}^2-\{\mathbf{P}\})$. By (\ref{piiszero}) this contribution is zero.
	
	Similarly for all other cases where the points are not in general position, there will be a factor $\pi_n(\mathbb{P}^1)$ with $n\geq 3$ or $\pi_n(\mathbb{P}^1-\{\mathbf{P}\})$ with $n\geq2$.
	So we only need to count the contribution of the general position case.
	To count how many $\lambda$-tuples of points are in general position we take all of $(\mathbb{P}^2-\{\mathbf{P}\})(\lambda)$ and subtract the other cases.
	After multiplying by $(-1)^{\sum_i\lambda_i}$ and summing over all $\lambda$ of weight $w$, we obtain $\pi_w(\mathbb{P}^2-\{\mathbf{P}\})$ minus terms that come from the other cases.
	The latter all have a factor zero.
	We conclude that (\ref{sieveeq}) can be rewritten as
	\begin{equation*}
	\mathcal{S}_{\text{split},5} = \sum_{i=1}^5\pi_{i}(\mathbb{P}^2-\{\mathbf{P}\})\cdot q^{15-3i} = q^{15}-q^{14}-q^{13}+q^{12}.
	\end{equation*}	
\end{proof}

\begin{lemma}
	\label{nonsplitresult}
	The sieve count for non-split nodes is
	\[\frac{\mathcal{S}_{\text{\emph{non-split}},5}}{2(q^4-q^2)} = \frac{1}{2}(q^{11}-q^{10}).\]
\end{lemma}
\begin{proof}
	Let $\lambda$ be a partition such that $|\lambda|\leq5$.
	If $S\in(\mathbb{P}^2-\{\mathbf{P}\})(\lambda)$ contains points on 1 of the 2 tangent lines through $\mathbf{P}$ then it also has to contain the conjugate points on the other tangent line.
	Because of this there are only few cases for which $\mathbf{C}_\text{non-split}(S)$ is nonempty.
	In those few cases we can apply a $\bar{k}$-linear transformation and use the dimension count from Lemma~\ref{sievetypes}.
	The rest of the argument is very similar to the proof of Lemma~\ref{splitresult} and we obtain $\mathcal{S}_{\text{split},5}=\mathcal{S}_{\text{non-split},5}$.
\end{proof}

\begin{lemma}
	\label{cuspresult}
	The sieve count for cusps is
	\[\frac{\mathcal{S}_{\text{\emph{cusp}},5}}{q^3(q-1)^2} = q^{10}-q^8.\]
\end{lemma}
\begin{proof}
	Besides the conditions on $\mathbf{P}$ and the double tangent line we have the extra condition that the coefficient of $x^3z^2$ must be nonzero.
	Because of this the configurations described in Lemma~\ref{sievetypes} are not sufficient.
	Let $S\in(\mathbb{P}^2-\{\mathbf{P}\})(\lambda)$ for some partition $\lambda$ with $|\lambda|\leq5$.
	A configuration of points $S\cup\{\mathbf{P}\}$ is no longer in general position when $S$ contains 2 points on $\{y=0\}$ or 4 points on an irreducible conic that has tangent $\{y=0\}$ at $\mathbf{P}$. The line $\{y=0\}$ or respectively the conic would have to be a component on any curve in $\mathbf{C}_\text{cusp}(S)$, which contradicts having an ordinary cusp at $\mathbf{P}$.
	
	The only configurations $S\cup\{\mathbf{P}\}$ for which $\mathbf{C}_\text{cusp}(S)$ is nonempty are the general position case and the case when there are 4 or 5 points on a line that does not pass through $\mathbf{P}$.
	For both cases we can compute the dimension of $\mathbf{C}_\text{cusp}(S)$ in the same way as in the proof of Lemma~\ref{sievetypes}.
	In all cases, except the general position case with $|S|=5$, this can be done independently from the coefficient of $x^3z^2$.
	So in these cases $\mathbf{C}_\text{cusp}(S)$ will be of the form $\mathbb{A}^d-\mathbb{A}^{d-1}$.
	When we look at all the ways we could have a non-ordinary cusp and 5 other singularities we see that the points will never be in general position.
	So for the general position case with $|S|=5$, the space $\mathbf{C}_\text{cusp}(S)$ is a point.
	
	The rest of the argument is very similar to the proof of Lemma~\ref{splitresult} and substituting terms of the form $q^d$ with $q^d-q^{d-1}$ in the expression for $\mathcal{S}_{\text{split},5}$ gives us the expression for $\mathcal{S}_{\text{cusp},5}$.
\end{proof}
    
\section{The case of more than 5 extra singularities.}
\label{explicitpart}
In this section we will compute the remaining part of the right hand side of (\ref{finalsieveeq}) (for fixed $N=5$), which is
\begin{equation}
\label{expsum}
\frac{1}{|\text{PGL}_3(k)|}\sum_{|\lambda|>5}\sum_{C\in \mathbf{C}([\lambda,1^1])}\delta_k^1(C)\cdot\sigma_5(\lambda).
\end{equation}
The contribution of a curve to this sum depends on $\delta_k^1(C)$ and on the fields its singularities are defined over.
This information is contained in the curve's type.
\begin{definition}
	We define the \textit{type} of a projective plane $k$-quintic to be given by the following information:
	\begin{itemize}
		\item the degree and multiplicity of all its irreducible components,
		\item the number of singularities of each irreducible component and their delta invariants,
		\item the number of points in which the subsets of irreducible components intersects,
		\item the fields over which the singularities and the irreducible components are defined.
	\end{itemize}
\end{definition}

We will compute the number of curves of each type where the singularities form a $[\lambda,1^1]$-tuple such that $|\lambda|>5$, $\sigma_5(\lambda)\neq0$ and there is at least 1 singularity of delta invariant 1.
\begin{lemma}
	\label{inftylemma}
	If $\lambda$ is an infinite partition such that there is a curve $C\in\mathbf{C}([\lambda,1^1])$ with $\delta_k^1(C)\neq0$, then $\sigma_5(\lambda)=0$.
\end{lemma}
\begin{proof}
	Because $C$ has an singularity of delta invariant 1, the double component of $C$ can only be 1 double (or triple) line.
	Let $\eta$ be the partition such that the singularities of $C$ that are not on the double line form a $\eta$-tuple of points.
	Let $\rho$ be the infinite partition such that $\rho_i$ is the number of conjugate $i$-tuples in $\mathbb{P}^1$, then we have:
	\begin{align*}
	\sigma_5(\lambda) &= -\sum_{\substack{\mu\subset\lambda \\ |\mu|\leq 5}}(-1)^{\sum_i \mu_i}\prod_{i=1}^v\binom{\lambda_i}{\mu_i} \\
	&= -\sum_{\substack{\nu\subset\eta \\ |\nu|\leq 5}}\sum_{\substack{\mu\subset\rho \\ |\mu|\leq 5-|\nu|}}(-1)^{\sum_i \mu_i+\nu_i}\prod_{i=1}^v\binom{\rho_i}{\mu_i}\binom{\eta_i}{\nu_i} \\
	&= -\sum_{\substack{\nu\subset\eta \\ |\nu|\leq 5}}(-1)^{\sum_i \nu_i}\prod_{i=1}^v\binom{\eta_i}{\nu_i}\sum_{|\mu|\leq 5-|\nu|}(-1)^{\sum_i \mu_i}|\mathbb{P}^1(\mu)|
	\end{align*}
	Since $|\nu|\leq|\eta|\leq3$, we can use (\ref{piiszero}) to deduce
	\[\sum_{|\mu|\leq 5-|\nu|}(-1)^{\sum_i \mu_i}|\mathbb{P}^1(\mu)|
	=\sum_{i=0}^{5-|\nu|}\pi_i(\mathbb{P}^1)
	=\sum_{i=0}^\infty\pi_i(\mathbb{P}^1)
	=\frac{1}{Z(\mathbb{P}^1;1)}
	=0.\]
\end{proof}
By Lemma \ref{inftylemma} we can assume that all curves in $\mathbf{C}([\lambda,1^1])$ have no irreducible components of higher multiplicity.

The contribution of the different types of curves is described in the tables of the following lemma.
The first columns describe the fields over which the components and/or points are defined.
The column marked $\delta_k^1$ contains the number of singularities of delta invariant 1.
The column marked $|\{C\}|$ contains the number of curves of each type divided by $|\text{PGL}_3(k)|$.
\begin{lemma} 
\label{explicitcounting}
 The contributions to Equation (\ref{expsum}) of the different types of curves is as follows:
 (The cases where $\sigma_5(\lambda)=0$ or $\delta_k^1=0$ are omitted.)
 \vskip 2mm
 \small Curves consisting of 5 lines contribute:\normalsize
 \vskip -2mm
	\[\small\begin{tabular}{cccc}
		\emph{lines} & \emph{points} & $\delta_k^1$ & $|\{C\}|$ \\
		\midrule
		\multicolumn{4}{c}{the lines are in general position:} \\
		$[1^5]$ & $[1^{10}]$ & $10$ & $\frac{1}{120}(q-2)(q-3)$ \\
		$[1^3,2^1]$ & $[1^4,2^3]$ & $4$ & $\frac{1}{12}q(q-1)$ \\
		$[1^2,3^1]$ & $[1^1,3^3]$ & $1$ & $\frac{1}{6}(q+1)q$ \\
		\multicolumn{4}{c}{3 lines intersect in the same point:} \\
		$[1^5]$ & $[1^8]$ & $7$ & $\frac{1}{12}(q-2)$ \\ 
		$[1^3,2^1]$ & $[1^4,2^2]$ & $3$ & $\frac{1}{4}q$ \\
	\end{tabular}\normalsize\]
	
 \small Curves consisting of a conic and 3 lines contribute:\normalsize
 \[\resizebox{.9\textwidth}{!}{\begin{tabular}{ccccc}
		\emph{lines} & \emph{points} & $\delta_k^1$ & $|\{C\}|$ \emph{for} $\text{char}(k)\neq2$ & $|\{C\}|$ \emph{for} $\text{char}(k)=2$\\
		\midrule
		\multicolumn{5}{c}{general position:} \\
		$[1^3]$ & $[1^9]$ & $9$ & $\frac{1}{48}(q^2-7q+11)(q-3)$ & $\frac{1}{48}(q-2)(q-4)^2$ \\
		& $[1^7,2^1]$ & $7$ & $\frac{1}{16}(q^2-3q+1)(q-1)$ & $\frac{1}{16}q(q-2)^2$ \\
		& $[1^5,2^2]$ & $5$ & $\frac{1}{16}(q^3-2q^2-1)$ & $\frac{1}{16}q^2(q-2)$ \\
		& $[1^3,2^3]$ & $3$ & $\frac{1}{48}(q^2-3q+1)(q-1)$ & $\frac{1}{48}q(q-2)^2$ \\
		$[1^1,2^1]$ & $[1^3,2^3]$ & $3$ & $\frac{1}{8}(q^2-q-1)(q-1)$ & $\frac{1}{8}q^2(q-2)$ \\
		& $[1^1,2^4]$ & $1$ & $\frac{1}{8}(q^2-q-3)(q-3)$ & $\frac{1}{8}(q^2-2q-4)(q-2)$ \\
		& $[1^3,2^1,4^1]$ & $3$ & $\frac{1}{8}(q^2-q+1)(q+1)$ & $\frac{1}{8}q^3$ \\
		& $[1^1,2^2,4^1]$ & $1$ & $\frac{1}{8}(q^2-q-1)(q-1)$ & $\frac{1}{8}q^2(q-2)$ \\
		\multicolumn{5}{c}{the 3 lines pass through 1 point:} \\
		$[1^3]$ & $[1^7]$ & $6$ & $\frac{1}{48}(q-3)^2$ & $\frac{1}{48}(q-2)(q-4)$ \\
		& $[1^5,2^1]$ & $4$ & $\frac{1}{16}(q-1)^2$ & $\frac{1}{16}q(q-2)$ \\
		& $[1^3,2^2]$ & $2$ & $\frac{1}{16}(q-1)^2$ & $\frac{1}{16}q(q-2)$ \\
		$[1^1,2^1]$ & $[1^3,2^2]$ & $2$ & $\frac{1}{8}(q^2-2q-1)$ & $\frac{1}{8}q(q-2)$ \\
		& $[1^3,4^1]$ & $2$ & $\frac{1}{8}(q+1)(q-1)$ & $\frac{1}{8}q^2$ \\
		\multicolumn{5}{c}{1 of the lines is tangent to the conic:} \\
		$[1^3]$ & $[1^8]$& $7$ & $\frac{1}{8}(q-3)^2$ & $\frac{1}{8}(q-2)(q-4)$ \\
		& $[1^6,2^1]$ & $5$ & $\frac{1}{4}(q-1)^2$ & $\frac{1}{4}q(q-2)$ \\
		& $[1^4,2^2]$ & $3$ & $\frac{1}{8}(q-1)^2$ & $\frac{1}{8}q(q-2)$ \\
		\multicolumn{5}{c}{2 of the lines are tangent to the conic:} \\
		$[1^3]$ & $[1^7]$ & $5$ & $\frac{1}{4}(q-3)$ & $\frac{1}{4}(q-2)$ \\
		& $[1^5,2^1]$ & $3$ & $\frac{1}{4}(q-1)$ & $\frac{1}{4}q$ \\
		$[1^1,2^1]$ & $[1^3,2^2]$ & $3$ & $\frac{1}{4}(q-1)$ & $\frac{1}{4}q$ \\
		& $[1^1,2^3]$ & $1$ & $\frac{1}{4}(q-3)$ & $\frac{1}{4}(q-2)$ \\
		\multicolumn{5}{c}{2 of the lines intersect on the conic:} \\
		$[1^3]$ & $[1^7]$ & $6$ & $\frac{1}{4}(q-2)(q-3)$ & $\frac{1}{4}(q-2)(q-3)$ \\
		& $[1^5,2^1]$ & $4$ & $\frac{1}{4}q(q-1)$ & $\frac{1}{4}q(q-1)$ \\
		$[1^1,2^1]$ & $[1^3,2^2]$ & $2$ & $\frac{1}{4}q(q-1)$ & $\frac{1}{4}q(q-1)$ \\
	\end{tabular}}\]

 \small Curves consisting of 2 conics and a line contribute:\normalsize
 \vskip -4mm
	 \[\resizebox{\textwidth}{!}{\begin{tabular}{ccccc}
		\emph{conic intersection} & \emph{points} & $\delta_k^1$ & $|\{C\}|$ \emph{for} $\text{char}(k)\neq2$ & $|\{C\}|$ \emph{for} $\text{char}(k)=2$\\
		\midrule
		\multicolumn{5}{c}{general position and the conics are defined over $k$:} \\
		$[1^4]$ & $[1^8]$ & $8$ & $\frac{1}{192}(q^2-9q+17)(q-3)(q-5)$ & $\frac{1}{192}(q-2)(q-4)(q-5)(q-6)$ \\
	 	& $[1^6,2^1]$ & $6$ & $\frac{1}{96}(q^2-5q+3)(q-1)(q-3)$ & $\frac{1}{96}q(q-2)(q-3)(q-4)$ \\
		& $[1^4,2^2]$ & $4$ & $\frac{1}{192}(q^2-q+1)(q-1)(q-3)$ & $\frac{1}{192}q(q-1)(q-2)^2$ \\		
		$[1^2,2^1]$ & $[1^6,2^1]$ & $6$ & $\frac{1}{32}(q^2-3q+1)(q-1)(q-3)$ & $\frac{1}{32}q(q-2)^2(q-3)$ \\
		& $[1^4,2^2]$ & $4$ & $\frac{1}{16}(q^3-2q^2-1)(q-1)$ & $\frac{1}{16}q^2(q-1)(q-2)$ \\		
		$[1^1,3^1]$ & $[1^5,3^1]$ & $5$ & $\frac{1}{24}(q+1)q(q-1)(q-2)$ & $\frac{1}{24}q(q+1)(q-1)(q-2)$ \\
		& $[1^3,2^1,3^1]$ & $3$ & $\frac{1}{12}q^2(q+1)(q-1)$ & $\frac{1}{12}q^2(q+1)(q-1)$ \\
		& $[1^1,2^2,3^1]$ & $1$ & $\frac{1}{24}(q+1)q(q-1)(q-2)$ & $\frac{1}{24}q(q+1)(q-1)(q-2)$ \\		
		$[2^2]$ & $[1^4,2^2]$ & $4$ & $\frac{1}{64}(q^3-2q^2-2q-1)(q-3)$ & $\frac{1}{64}q(q^2-3q-2)(q-2)$ \\		
		$[4^1]$ & $[1^4,4^1]$ & $4$ & $\frac{1}{32}(q^2-q-1)(q+1)(q-1)$ & $\frac{1}{32}q^2(q+1)(q-2)$ \\
		\multicolumn{5}{c}{general position and the conics form a conjugate pair:} \\
		$[1^4]$ & $[1^4,2^2]$ & $4$ & $\frac{1}{96}(q^3-4q^2+4q+3)(q-1)$ & $\frac{1}{96}q(q-1)(q-2)^2$ \\
		& $[1^4,4^1]$ & $4$ & $\frac{1}{96}(q^2-3q+3)(q+1)(q-1)$ & $\frac{1}{96}q^2(q-1)(q-2)$ \\
		$[1^1,3^1]$ & $[1^1,2^2,3^1]$ & $1$ & $\frac{1}{12}(q+1)q(q-1)(q-2)$ & $\frac{1}{12}(q+1)q(q-1)(q-2)$ \\
		& $[1^1,3^1,4^1]$ & $1$ & $\frac{1}{12}(q+1)q^2(q-1)$ & $\frac{1}{12}(q+1)q^2(q-1)$ \\
		\multicolumn{5}{c}{the line is tangent to 1 of the conics:} \\
		$[1^4]$ & $[1^7]$ & $6$ & $\frac{1}{48}(q-3)(q-4)(q-5)$ & $\frac{1}{48}(q-2)(q-4)(q-6)$ \\
		& $[1^5,2^1]$ & $4$ & $\frac{1}{48}(q-1)(q-2)(q-3)$ & $\frac{1}{48}q(q-2)(q-4)$ \\
		$[1^2,2^1]$ & $[1^5,2^1]$ & $4$ & $\frac{1}{8}(q-1)^2(q-2)$ & $\frac{1}{8}q(q-2)^2$ \\
		& $[1^3,2^2]$ & $2$ & $\frac{1}{8}q(q-1)^2$ & $\frac{1}{8}q^2(q-2)$ \\
		$[1^1,3^1]$ & $[1^4,3^1]$ & $3$ & $\frac{1}{6}(q+1)q(q-1)$ & $\frac{1}{6}(q+1)q(q-1)$ \\
		& $[1^2,2^1,3^1]$ & $1$ & $\frac{1}{6}(q+1)q(q-1)$ & $\frac{1}{6}(q+1)q(q-1)$ \\
		$[2^2]$ & $[1^3,2^2]$ & $2$ & $\frac{1}{16}q(q-1)(q-3)$ & $\frac{1}{16}q(q-2)^2$ \\
		$[4^1]$ & $[1^3,4^1]$ & $2$ & $\frac{1}{8}(q+1)q(q-1)$ & $\frac{1}{8}q^3$ \\
		\multicolumn{5}{c}{the conics are defined over $k$ and intersect in 3 points:} \\
		$[1^3]$ & $[1^7]$ & $6$ & $\frac{1}{16}(q-3)^2(q-5)$ & $\frac{1}{16}(q-2)(q-4)(q-5) $ \\
		& $[1^5,2^1]$ & $4$ & $\frac{1}{8}(q-1)^2(q-3)$ & $\frac{1}{8}q(q-2)(q-3)$ \\
		& $[1^3,2^2]$ & $2$ & $\frac{1}{16}(q-1)^3$ & $\frac{1}{16}q(q-1)(q-2)$ \\
		$[1^1,2^1]$ & $[1^5,2^1]$ & $4$ & $\frac{1}{16}(q-1)^2(q-3)$ & $\frac{1}{16}q(q-2)(q-3)$ \\
		& $[1^3,2^2]$ & $2$ & $\frac{1}{8}(q-1)^3$ & $\frac{1}{8}q(q-1)(q-2)$ \\
		\multicolumn{5}{c}{the conics form a conjugate pair and intersect in 3 points:} \\
		$[1^1,2^1]$ & $[1^3,2^2]$ & $2$ & $\frac{1}{8}(q^2-2q-1)(q-1)$ & $\frac{1}{8}q(q-1)(q-2)$ \\
		& $[1^3,4^1]$ & $2$ & $\frac{1}{8}(q+1)(q-1)^2$ & $\frac{1}{8}q^2(q-1)$ \\ 
	 \end{tabular}}\]
 
   \small Curves consisting of a cubic and 2 lines contribute: \normalsize
 \[\small\begin{tabular}{cccc}
 \emph{lines} & \emph{points} & $\delta_k^1$ & $|\{C\}|$ \\
 \midrule
 \multicolumn{4}{c}{a nonsingular cubic:} \\
 $[1^2]$ & $[1^7]$ & $7$ & $\frac{1}{72}(q^4-3q^3+3q^2-17q+36)(q-2)$ \\
 & $[1^5,2^1]$ & $5$ & $\frac{1}{12}(q^3-q^2+q-3)q(q-2)$ \\
 & $[1^3,2^2]$ & $3$ & $\frac{1}{8}(q^2-q+2)(q+1)q(q-1)$ \\
 & $[1^4,3^1]$ & $4$ & $\frac{1}{18}(q^2-q+1)(q+1)q(q-2)$ \\
 & $[1^1,3^2]$ & $1$ & $\frac{1}{18}(q^3-2)(q+1)q$ \\
 & $[1^2,2^1,3^1]$ & $2$ & $\frac{1}{6}(q^2-q+1)(q+1)q^2$ \\
 $[2^1]$ & $[1^1,2^3]$ & $1$ & $\frac{1}{12}(q^3-q^2+4)(q^2-3)$ \\
 & $[1^1,2^1,4^1]$ & $1$ & $\frac{1}{4}(q+1)q^2(q-1)^2$ \\
 & $[1^1,6^1]$ & $1$ & $\frac{1}{6}(q^3-q^2-2)q^2$ \\
 \multicolumn{4}{c}{a singular cubic:} \\
 $[1^2]$ & $[1^8]$ & $8$ & $\frac{1}{72}(q-2)(q-3)(q-4)(q-5)$ \\
 & $[1^6,2^1]$ & $6$ & $\frac{1}{12}q(q-1)(q-2)(q-3)$ \\
 & $[1^4,2^2]$ & $4$ & $\frac{1}{8}(q+1)q(q-1)(q-2)$ \\
 & $[1^5,3^1]$ & $5$ & $\frac{1}{18}(q+1)q(q-1)(q-2)$ \\
 & $[1^3,2^1,3^1]$ & $3$ & $\frac{1}{6}(q+1)q^2(q-1)$ \\
 \multicolumn{4}{c}{a singular cubic and 1 line is tangent to it:} \\
 $[1^2]$ & $[1^7]$ & $6$ & $\frac{1}{6}(q-2)(q-3)(q-4)$ \\
 & $[1^5,2^1]$ & $4$ & $\frac{1}{2}q(q-1)(q-2)$ \\
 & $[1^4,3^1]$ & $3$ & $\frac{1}{3}(q+1)q(q-1)$ \\
 \end{tabular}\normalsize\]
 \newpage
 
 \small Curves consisting of a cubic and a conic contribute: \normalsize
 \vskip -3mm
 \[\resizebox{\textwidth}{!}{\begin{tabular}{cccc}
 	\emph{points} & $\delta_k^1$ & $|\{C\}|$ \emph{for} $\text{char}(k)\neq2$ & $|\{C\}|$ \emph{for} $\text{char}(k)=2$\\
 	\midrule
 	$[1^7]$ & $7$ & $\frac{1}{720}(q^2-9q+15)(q-3)(q-5)(q-7)$ & $\frac{1}{720}(q-2)(q-4)^2(q-6)(q-8)$ \\
 	$[1^5,2^1]$ & $5$ & $\frac{1}{48}(q^3-6q^2+10q-1)(q-1)(q-3)$ & $\frac{1}{48}(q-2)^3q(q-4)$ \\
 	$[1^3,2^2]$ & $3$ & $\frac{1}{16}(q^2-q-1)(q+1)(q-1)(q-3)$ & $\frac{1}{16}(q^2-2q-4)q^2(q-2)$ \\
 	$[1^1,2^3]$ & $1$ & $\frac{1}{48}(q^2-q-3)(q^2-2q-7)(q-3)$ & $\frac{1}{48}(q^2-2q-4)(q+2)(q-2)(q-4)$ \\
 	$[1^4,3^1]$ & $4$ & $\frac{1}{18}(q^2-3q+3)(q+1)q(q-1)$ & $\frac{1}{18}(q+1)q(q-1)^2(q-2)$ \\
 	$[1^1,3^2]$ & $1$ & $\frac{1}{18}(q^2+2q+3)q^2(q-2)$ & $\frac{1}{18}(q^2+2q+3)(q+1)(q-1)(q-2)$ \\
 	$[1^2,2^1,3^1]$ & $2$ & $\frac{1}{6}(q^2-q-1)(q+1)q(q-1)$ & $\frac{1}{6}(q+1)^2q(q-1)(q-2)$ \\
 	$[1^3,4^1]$ & $3$ & $\frac{1}{8}(q^2-q+1)(q+1)^2(q-1)$ & $\frac{1}{8}(q^2-2)q^3$ \\
 	$[1^1,2^1,4^1]$& $1$ & $\frac{1}{8}(q^2-q-1)(q+1)(q-1)^2$ & $\frac{1}{8}(q^2-2)q^2(q-2)$ \\
 	$[1^2,5^1]$ & $2$ & $\frac{1}{5}(q^2+1)q^2(q+1)$ & $\frac{1}{5}(q^2+1)(q+1)^2(q-1)$ \\
 	$[1^1,6^1]$ & $1$ & $\frac{1}{6}(q^2+q+2)q^2(q-1)$ & $\frac{1}{6}(q^2+q+2)(q+1)(q-1)^2$ \\
 	\end{tabular}}\]
 \vskip 5mm
 \small Curves consisting of a quartic and a line contribute:\normalsize
 \vskip -3mm
 \[\small\begin{tabular}{cccc}
 \emph{sing. of quadric} & \emph{points} & $\delta_k^1$ & $|\{C\}|$ \\
 \midrule
 $[1^3]$ & $[1^7]$ & $7$ & $\frac{1}{144}(q-2)(q-3)^2(q-4)(q-5)$ \\
 & $[1^5,2^1]$ & $5$ & $\frac{1}{24}q(q-1)^2(q-2)(q-3)$ \\
 & $[1^3,2^2]$ & $3$ & $\frac{1}{48}(q+1)q(q-1)(q-2)(q-3)$ \\
 & $[1^4,3^1]$ & $4$ & $\frac{1}{18}(q+1)q^2(q-1)(q-2)$ \\
 & $[1^3,4^1]$ & $3$ & $\frac{1}{24}(q+1)q^2(q-1)^2$ \\
 $[1^1,2^1]$ & $[1^5,2^1]$ & $5$ & $\frac{1}{48}q(q-1)(q-2)(q-3)^2$ \\
 & $[1^3,2^2]$ & $3$ & $\frac{1}{8}(q+1)q(q-1)^2(q-2)$ \\
 & $[1^1,2^3]$ & $1$ & $\frac{1}{16}(q^2-q-4)(q+1)(q-2)(q-3)$ \\
 & $[1^2,2^1,3^1]$ & $2$ & $\frac{1}{6}(q+1)q^3(q-1)$ \\
 & $[1^1,2^1,4^1]$ & $1$ & $\frac{1}{8}(q+1)q^2(q-1)^2$ \\
 $[3^1]$ & $[1^4,3^1]$ & $4$ & $\frac{1}{72}(q+1)q(q-1)(q-2)(q-3)$ \\
 & $[1^2,2^1,3^1]$ & $2$ & $\frac{1}{12}(q+1)q^2(q-1)^2 $ \\
 & $[1^1,3^2]$ & $1$ & $\frac{1}{9}(q^3-q-3)(q+1)q$ \\
 \end{tabular}\normalsize\]

\vskip 2mm
\end{lemma}
\begin{proof}
	We will illustrate the methods used to prove the lemma by showing the calculations for a select few cases, the other cases are computed in similar ways. In all cases we have chosen all the points to be defined over $k$ in order to simplify the notation.
	The computations can easily be adapted to points over field extensions.
\\
\\
\textbf{Five $k$-lines in general position.}
\\
	First we take 2 $k$-points and then through each of these points a pair of $k$-lines such that none of the lines passes through both points.
	This can be done in $\binom{q^2+q+1}{2}\binom{q}{2}^2$ ways.
	The 4 lines have 6 intersection points.
	Take as the fifth line any line over $k$ that does not pass through any of these points.
	There are $(q-2)(q-3)$ such lines.
	We can end up with the same 5 lines if we start with another choice of 2 of the intersection points such that none of the 5 lines passes through both points.
	There are $15$ such choices, so we obtain
	\[\frac{1}{|\text{PGL}_3(k)|}\cdot\frac{1}{15}\binom{q^2+q+1}{2}\binom{q}{2}^2(q-2)(q-3)=\frac{1}{120}(q-2)(q-3).\]
	\vskip 2mm \noindent
\textbf{One conic and three $k$-lines with two of the lines tangent to the conic.}
\\
	The plane conics form a $\mathbb{P}^5$ and every reducible $k$-conic is either a pair of different $k$-lines, a conjugate pair of lines or a double $k$-line.
	So the number of irreducible plane conics over $k$ is given by
	\[\frac{q^6-1}{q-1}-\binom{q^2+q+1}{2}-\frac{1}{2}(q^4-q)-(q^2+q+1)=(q^2+q+1)q^2(q-1).\]
	Pick 2 $k$-points $A,B$ on the conic and take the tangent lines at these points, they meet in a point $P$.
	Choose 2 more $k$-points on the conic and take the line through them.
	We need to subtract the case where $P$ is on all 3 lines.
	If $\text{char}(k)\neq2$ then the lines through $A$ and $B$ are the only lines through $P$ that are tangent to the conic.
	There are $q-1$ other $k$-points on the conic that each determine a line through $P$.
	Each line is selected twice this way, so we obtain
	\[\frac{1}{|\text{PGL}_3(k)|}\cdot(q^2+q+1)q^2(q-1)\binom{q+1}{2}\left(\binom{q-1}{2}-\frac{q-1}{2}\right)=\frac{1}{4}(q-3).\]
	If $\text{char}(k)=2$ then all lines through $P$ are tangent to the conic, so we obtain
	\[\frac{1}{|\text{PGL}_3(k)|}\cdot(q^2+q+1)q^2(q-1)\binom{q+1}{2}\binom{q-1}{2}=\frac{1}{4}(q-2).\]
\textbf{Two $k$-conics and one line in general position.}
\\
	This is the most complicated case.
	Assume that $\text{char}(k)\neq2$.
	There are $\frac{1}{6}\binom{q^2+q+1}{2}(q^4-2q^3+q^2)$ ways to pick 4 $k$-points $P_1,P_2,P_3,P_4$ such that there are no 3 on a line.
	We write $\mathcal{P}$ for the pencil of conics through $P_1,P_2,P_3,P_4$.
	There are 6 lines through $P_1,P_2,P_3,P_4$ and these lines intersect in 7 $k$-points.
	Let $Q_1,Q_2,Q_3$ denote the other 3 intersection points.
	
	Let $L$ be any line such that none of $P_1,P_2,P_3,P_4$ is on $L$.
	We have a degree 2 map $\mathbb{P}^1\rightarrow\mathbb{P}^1$ that sends a point $R$ on $L$ to the conic in $\mathcal{P}$ through $R$.
	By Hurwitz's theorem there are 2 branch points, we call these 2 conics the \textit{tangent conics} to $L$.
	There are $3$ $k$-lines passing through 2 of $Q_1,Q_2,Q_3$.
	Let $L$ be one such a line.
	Both tangent conics for $L$ are reducible and each one intersects $L$ in 1 of the $Q_i$.
	There are $q-1$ other $k$-points on $L$ so there are $\frac{q-1}{2}$ conics in $\mathcal{P}(k)$ that intersect $L$ in 2 $k$-points.
	Exactly 1 of these conics is reducible so we have $\frac{q-3}{2}$ irreducible conics.
	
	There are $3(q-3)$ $k$-lines passing through precisely 1 of $Q_1,Q_2,Q_3$ and not through any of $P_1,P_2,P_3,P_4$.
	Let $L$ be such a line.
	One of the tangent conics for $L$ intersects $L$ in the $Q_i$ point so the other tangent conic for $L$ is also defined over $k$.
	There are $\frac{q-5}{2}$ irreducible conics in $\mathcal{P}(k)$ that intersect $L$ in 2 $k$-points.
	
	There are $q-2$ irreducible conics in $\mathcal{P}(k)$ that each have $q-3$ $k$-points besides $P_1,P_2,P_3,P_4$.
	This means we have
	\[\frac{1}{2}((q-2)(q-3)-3(q-3))=\frac{1}{2}(q-5)(q-3)\]
	$k$-lines that have 2 irreducible tangent conics over $k$.
	For every such line there are $\frac{q-7}{2}$ irreducible conics in $\mathcal{P}(k)$ that intersect it in 2 $k$-points.
	
	There are $(q-3)^2$ $k$-lines not passing through any of $P_1,P_2,P_3,P_4,Q_1,Q_2,Q_3$.
	So there are 
	\[(q-3)^2-\frac{1}{2}(q-5)(q-3)=\frac{1}{2}(q-3)(q-1)\]
	$k$-lines that have a conjugate pair of irreducible tangent conics.
	For every such line there are $\frac{q-5}{2}$ irreducible conics in $\mathcal{P}(k)$ that intersect it in 2 $k$-points.
	
	Putting everything together gives us
	\begin{multline*}
	\frac{1}{|\text{PGL}_3(k)|}\frac{1}{6}\binom{q^2+q+1}{2}(q^4-2q^3+q^2)
	\Bigg(
	3\binom{\frac{q-3}{2}}{2}
	+3(q-3)\binom{\frac{q-5}{2}}{2} \\
	+\frac{1}{2}(q-3)(q-5)\binom{\frac{q-7}{2}}{2}
	+\frac{1}{2}(q-1)(q-3)\binom{\frac{q-5}{2}}{2}
	\Bigg)
	=\frac{1}{192}(q^2-9q+17)(q-3)(q-5).
	\end{multline*}
	Now for the case where $\text{char}(k)=2$:
	Choose $P_1,P_2,P_3,P_4$ and define $\mathcal{P}$ and $Q_1,Q_2,Q_3$ as before.
	There is 1 line through $Q_1,Q_2,Q_3$ and it intersects every conic in $\mathcal{P}$ in 1 point.
	For any other line not through any of $P_1,P_2,P_3,P_4$ there is precisely 1 conic in $\mathcal{P}$ that intersects it in 1 point.
	
	There are $3(q-2)$ $k$-lines through 1 of $Q_1,Q_2,Q_3$ and not through any of $P_1,P_2,P_3,P_4$.
	For any such line there are $\frac{q-4}{2}$ irreducible conics in $\mathcal{P}(k)$ that intersect it in 2 $k$-points.
	
	There are $(q-2)(q-4)$ $k$-lines not through any of $P_1,P_2,P_3,P_4,Q_1,Q_2,Q_3$.
	For any such line there are $\frac{q-6}{2}$ irreducible conics in $\mathcal{P}(k)$ that intersect it in 2 $k$-points.
	
	Putting these two cases together we obtain
	\begin{multline*}
	\frac{1}{|\text{PGL}_3(k)|}\frac{1}{6}\binom{q^2+q+1}{2}(q^4-2q^3+q^2)
	\Bigg(
	3(q-2)\binom{\frac{q-4}{2}}{2}
	+(q-2)(q-4)\binom{\frac{q-6}{2}}{2}
	\Bigg) \\
	=\frac{1}{192}(q-2)(q-4)(q-5)(q-6).
	\end{multline*}

\noindent
\textbf{A singular cubic and two $k$-lines.}
\\
	Take a $k$-point $P$ and 2 $k$-lines not passing through $P$.
	Then choose 3 $k$-points on each line, such that no 2 of these 6 points are on a line through $P$ and none of the points is the intersection of the 2 lines.
	There is precisely 1 irreducible cubic through the 6 points that has a singularity at $P$ of delta invariant 1.
	We obtain
	\[\frac{1}{|\text{PGL}_3(k)|}\binom{q^2+q+1}{2}(q^2-q)\binom{q}{3}\binom{q-3}{3}=\frac{1}{72}(q-2)(q-3)(q-4)(q-5).\]
\vskip 2mm
\noindent
\textbf{A nonsingular cubic and two $k$-lines.}
\\
	Take 2 $k$-lines $L,L'$ and 3 $k$-points on each line that are not the intersection point.
	There is a $\mathbb{P}^3$ of cubics through the 6 points.
	From this we subtract the reducible cubics through the 6 points.
	We get a reducible cubic by taking any $k$-line together with the lines $L,L'$.
	Or we can take a line though 2 of the points and an irreducible conic through the other 4 points.
	Finally we can take 3 lines connecting the 6 points such that none of the lines is $L$ or $L'$.
	
	We also subtract the singular irreducible cubics we counted above.
	This gives us
	\begin{multline*}
	\frac{1}{|\text{PGL}_3(k)|}\binom{q^2+q+1}{2}\binom{q}{3}^2
	\left(|\mathbb{P}^3|-|\mathbb{P}^2|-9(q-2)-6\right) \\
	-\frac{1}{72}(q-2)(q-3)(q-4)(q-5)
	=\frac{1}{72}(q^4-3q^3+3q^2-17q+36)(q-2).\qedhere 
	\end{multline*}
\end{proof}

By putting together the information collected in Lemma~\ref{explicitcounting} we deduce:
\begin{lemma}
	\label{explicitresult}
\[\frac{1}{|\text{PGL}_3(k)|}\sum_{|\lambda|>5}\sum_{C\in \mathbf{C}([\lambda,1^1])}\delta_k^1(C)\cdot\sigma_5(\lambda)=1\]
\end{lemma}
\begin{proof}
	By Lemma~\ref{inftylemma} the types of curves that contribute to this sum are all listed in Lemma~\ref{explicitcounting}.
	We compute the following sum over the rows of the tables, where we consider each type of curve to have a $[\lambda,1^1]$-tuple of singular points.
	For all characteristics we obtain the same result
	\[\sum \sigma_5(\lambda)\cdot\delta_k^1(C)\cdot|\{C\}|=1.\]
\end{proof}

\section{Combining results and verification.}
\label{results}

Applying Lemmas \ref{splitresult}, \ref{nonsplitresult}, \ref{cuspresult} and \ref{explicitresult} to Equation (\ref{finalsieveeq}) we obtain
\[|\mathcal{T}_5(k)|=\frac{1}{2}(q^{11}+q^{10})+\frac{1}{2}(q^{11}-q^{10})+q^{10}-q^8+1,\]
which proves our main result:
\begin{theorem}
\label{maintheorem}
The number of smooth trigonal curves of genus $5$ over a finite field $\mathbb{F}_q$, weighted by the size of their automorphism group, is given by
\[|\mathcal{T}_5(\mathbb{F}_q)|=q^{11}+q^{10}-q^{8}+1.\]
\end{theorem}
As an extra check we have written computer programs that loop over all plane quintics over $\mathbb{F}_2$ and $\mathbb{F}_3$ that have an ordinary split node/non-split node/cusp with fixed tangents at $\mathbf{P}$.
For each curve we test for all points besides $\mathbf{P}$ whether they are singular or not.
This way we managed to count
\begin{equation}
\label{result1}
|\mathbf{C}_\text{split}(\mathbf{P},\lambda)|,\quad
|\mathbf{C}_\text{non-split}(\mathbf{P},\lambda)|\quad\text{and}\quad
|\mathbf{C}_\text{cusp}(\mathbf{P},\lambda)|
\end{equation}
for all partitions $\lambda$.
From the results for the empty partition we can now easily compute $|\mathcal{T}_5(\mathbb{F}_2)|$ and $|\mathcal{T}_5(\mathbb{F}_3)|$.
As an extra check the programs use the information from (\ref{result1}) to count
\[\sum_{|\lambda|=w}\left((-1)^{\sum_i \lambda_i}\cdot\sum_{S\in(\mathbb{P}^2-\{\mathbf{P}\})(\lambda)}|\mathbf{C}_\text{split}(S)|\right)\]
for all $0\leq w\leq5$ and 
\[\sum_{|\lambda|=w}|\mathbf{C}_\text{split}(\mathbf{P},\lambda)|\cdot \sigma_N(\lambda)\]
for all $6\leq w\leq9$.
The results of these computer counts agree with the counts in this article.
The programs are written in the C programming language and the source code is available at
github.com/Wennink/countingtrigonalcurves,
together with lists of the results for (\ref{result1}) and a comparison with results from this article.


\end{document}